%% file: main.tex
\documentclass[letterpaper, 10 pt, conference]{ieeeconf}

\IEEEoverridecommandlockouts
\usepackage{cite}

\usepackage{amsmath,amsthm,amssymb,amsfonts}
\allowdisplaybreaks
\usepackage{algorithmic}
\usepackage{graphicx}
\usepackage{textcomp}
\usepackage{xcolor}
\usepackage{bbm}
\usepackage{balance}
\usepackage{bm}
\usepackage{algorithmic}
\usepackage{algorithm}
\usepackage[caption=false, font=footnotesize]{subfig}
\newtheorem{theorem}{Theorem}
\newtheorem{definition}{Definition}

\newtheorem{assumption}{Assumption}
\newtheorem{proposition}{Proposition}
\newtheorem{remark}{Remark}

\ifodd 0

\newcommand{\wjcom}[1]{\textbf{\color{white!50!blue} (Wenjie comment: #1)}} %comment of the text
\else

\newcommand{\wjcom}[1]{}

\fi

\def\BibTeX{{\rm B\kern-.05em{\sc i\kern-.025em b}\kern-.08em
    T\kern-.1667em\lower.7ex\hbox{E}\kern-.125emX}}

\definecolor{lightgreen}{RGB}{225, 245, 230}
\begin{document}

% \title{Consensus-Based Distributed Feedback Optimization with Gradient Estimation\\
%On Distributed Constrained Feedback Optimization: A Projected Model-Free Method

\title{\LARGE \bf Robust Error Bounds for Vector-Valued Kernel Ridge Regression
\thanks{
The authors are with Automatic Control Laboratory, EPFL, Switzerland. Email: wenbin.wang@epfl.ch, colin.jones@epfl.ch. This work was supported by the Swiss Federal Office of Energy SFOE as part of the SWEET consortium SWICE.}
}
%New command
\newcommand{\obj}{\Tilde{\Phi}}
\newcommand{\gest}{\Tilde{\nabla}}
\newcommand{\tr}[1]{\text{tr}(#1)}
\newcommand{\E}[1]{\bb{E}[#1]}
\newcommand{\Ef}[2]{\bb{E}[#1|\mathcal{F}_{#2}]}
\newcommand{\N}{\frac{1}{N}\mathbf{1}\mathbf{1}^{\top}}
\newcommand{\inff}[1]{#1_{\infty}}
\newcommand{\bb}[1]{\mathbb{#1}}
\newcommand{\norm}[1]{\|#1\|}
\newcommand{\innerp}[3]{\langle #2, #3 \rangle_{\mathcal{#1}}}

\author{Wenbin Wang, Colin N. Jones}

\maketitle

\input{section/0_abstract}
\input{section/1_intro}
\input{section/2_preliminary}
\input{section/3_error_bounds}
% \input{section/4_efficient_computation}
\input{section/5_numerical_result}
\input{section/6_conclusion}

\balance
\bibliographystyle{IEEEtran}
\bibliography{references}

\end{document}

%% file: section/0_abstract.tex
\begin{abstract}
While kernel-based learning methods for scalar-valued functions have been extensively studied, comparatively little work considers nonlinear vector-valued operators. Such operators with vector-valued outputs, e.g., state vectors and trajectories, commonly arise in system identification problems. In this paper, we study the problem of learning a vector-valued operator under the framework of vector-valued kernel ridge regression. We assume that the outputs take values in an arbitrary Hilbert space and are corrupted by additive noise with bounded norm. Under the assumption that the unknown operator belongs to a vector-valued reproducing kernel Hilbert space, we derive a pointwise deterministic bound on the discrepancy between the kernel ridge regression estimator and the ground truth in the corresponding Hilbert space norm. Numerical simulations under two different scenarios are presented to validate the theoretical results.
\end{abstract}

%% file: section/1_intro.tex
\section{Introduction}
Kernel-based learning methods for system identification have attracted significant attention in recent years. Classical approaches typically rely on parametric model structures, e.g., AutoRegressive with eXogenous input (ARX) model, whose parameters are estimated via basic methods such as the least-squares method \cite{pillonetto2014kernel}. While these approaches are simple to implement and computationally efficient, they suffer from several limitations, including sensitivity to model order selection and high estimation variance under poor excitation conditions \cite{10266828}. In contrast, kernel-based methods adopt a nonparametric framework, where the system dynamics are learned in a high-dimensional (possibly infinite-dimensional) function space and the model complexity is managed through regularization. These methods have demonstrated strong performance in both linear \cite{pillonetto2024kernel,dinuzzo2015kernels} and nonlinear system identification \cite{lian2021nonlinear,lian2020gaussian}. Moreover, kernel-based regularization provides principled and tight uncertainty quantification for the identified dynamics \cite{scampicchio2025gaussian} when interpreted from a Bayesian perspective. This capability has made it particularly attractive for learning-based control, including robust control \cite{7330913} and model predictive control \cite{scampicchio2025gaussian}.

However, existing kernel-based system identification methods primarily focus on learning scalar-valued functions~\cite{lian2020gaussian,scampicchio2025gaussian}, which limits their ability to handle multi-input multi-output (MIMO) systems.
A common alternative is to perform identification separately for each output dimension~\cite{scampicchio2025gaussian}, which implicitly assumes independence among the output dimensions. However, this assumption is rarely satisfied in practice. In many MIMO systems, such as power systems~\cite{zhao2015distributed} and networked dynamical systems~\cite{nedic2018network}, outputs are coupled through underlying physical interactions. Ignoring these couplings discards shared information, leading to data-inefficient learning and degraded estimation performance. 

Moreover, this framework cannot accommodate cases where the outputs are high-dimensional vectors or even infinite-dimensional objects. For systems governed by partial differential equations (PDEs), measurements often take the form of spatially distributed fields~\cite{li2020fourier}, such as temperature, pressure, or velocity profiles. Here, the output resides in an infinite-dimensional function space, and classical scalar-valued kernels are fundamentally inadequate for representing the corresponding input–output mappings. These considerations motivate an extension of the classical scalar-valued framework to the vector-valued setting. Such an extension enables the joint learning of coupled outputs and provides a principled framework for modeling structured, potentially infinite-dimensional system responses.

Problems involving general vector-valued or function-valued outputs are commonly referred to as functional data analysis (FDA) \cite{gertheiss2024functional,kadri2010nonlinear} in the statistics community. Within this framework, various forms of functional regression have been studied, including scalar-on-function \cite{yuan2010reproducing}, function-on-scalar \cite{chen2016variable}, and function-on-function regression paradigms \cite{lian2020gaussian, bevanda2023koopman}. Some of the popular nonlinear system identification methods, such as the Koopman operator, fall in the last category. In the machine learning community, a closely related problem is often referred to as the multi-task learning problem \cite{ruiz2024survey}, where the objective is to learn multiple scalar-valued functions simultaneously while exploiting potential correlations among them. This formulation has been the subject of extensive theoretical analysis, including results on convergence rates \cite{caponnetto2007optimal} and algorithmic stability of multi-task kernel regression \cite{audiffren2013stability}. These theoretical developments have enabled practical applications such as multi-fidelity modeling \cite{park2017remarks}, transfer learning across outputs \cite{pan2009survey}, and learning from heterotopically sampled data \cite{liu2018remarks}. While most works on vector-valued regression focus on minimizing the risk functional and deriving the corresponding estimator, few provide theoretical analysis of the estimation error \cite{kadri2009general}. 

In this work, we consider the problem of learning a vector-valued operator that maps from a compact input space to a separable Hilbert space. The observations are elements of the corresponding Hilbert space with additive bounded-norm noise that resides in the same space. We assume that the unknown operator belongs to a vector-valued reproducing kernel Hilbert space (RKHS) with a known norm bound. The estimator that minimizes the regularized risk functional is obtained by solving a vector-valued kernel ridge regression problem. For each input, we derive an explicit upper bound on the normed distance between the model prediction and the ground truth. The proposed theoretical results are validated through numerical experiments on operators with two types of output spaces: a two-dimensional Euclidean space, and an infinite-dimensional scalar-valued RKHS.

% Some theoretical work emphasizes asymptotic behavior, establishing optimality or consistency only in the limit of an infinite number of samples. However, such assumptions are rarely satisfied in many applications in control. 

%% file: section/2_preliminary.tex
\section{preliminaries and problem formulation}
\subsection{Notation}
We denote the set of positive integers by $\mathbb{N} = \{1,2,\dots\}$. For $m\in\mathbb{N}$, we write $[m]:=\{1,\dots,m\}$. Let $\mathbb{R}^n$ be the $n$-dimensional Euclidean space. Given a vector $\theta\in\mathbb{R}^n$, its $i$-th entry is $\theta(i)$. We denote the vector-valued reproducing kernel Hilbert space (RKHS) by $\mathcal H$, which is equipped with the inner product $\langle \cdot,\cdot\rangle_{\mathcal{H}}$ and the norm $\|\cdot\|_{\mathcal{H}}$. Let $\mathcal{X}$ be the input space. A mapping $f : \mathcal X \to \mathcal Y$ is referred to as a vector-valued operator if, for every $x \in \mathcal X$, the output satisfies $f(x) \in \mathcal Y$, where $\mathcal Y$ is a Hilbert space endowed with an inner product
$\langle \cdot, \cdot \rangle_{\mathcal Y}$ and norm
$\|\cdot\|_{\mathcal Y}$. We denote an $n$-dimensional subspace of $\mathcal Y$ by $\mathcal Y_n$, and the direct sum of $m$ copies of $\mathcal Y$ by $\mathcal Y^m$. Furthermore, we denote the space of bounded linear operators mapping one Hilbert space $\mathcal{Y}_1$ to another Hilbert space $\mathcal{Y}_2$ by $\mathcal{L}(\mathcal{Y}_1,\mathcal{Y}_2)$. In the special case where $\mathcal{Y}_1 = \mathcal{Y}_2 = \mathcal{Y}$, we write $\mathcal L(\mathcal Y)$ for brevity. We also denote the cone of positive semidefinite (self-adjoint) bounded linear operators by $\mathcal L_{+}(\mathcal Y) \subset \mathcal L(\mathcal Y)$. For any bounded linear operator $T$ between two Hilbert spaces, we denote its adjoint by $T^*$.

\subsection{Vector-Valued RKHS}
Here, we adopt the definition of vector-valued RKHS from~\cite{micchelli2005learning}.

\begin{definition}
    $\mathcal{H}$ is a vector-valued reproducing kernel Hilbert space if for all $y\in\mathcal{Y}$ and $x \in \mathcal{X}$, the linear functional mapping $f$ to $\langle y,f(x)\rangle_{\mathcal{Y}}$ is continuous.
\end{definition}

\begin{remark}
    For a scalar-valued RKHS, the defining property is that the point evaluation functional
\[
f \mapsto f(x)
\]
is continuous. In the vector-valued setting, it is natural to pair $f(x)$ with an element $y \in \mathcal Y$ via the inner product, as it provides the canonical way to transform the vector-valued evaluation $f(x)$ into a scalar while preserving linearity. Such an operation is linear both in $f$ and $y$, which gives rise to operator-valued reproducing kernels.
\end{remark}
By the Riesz representation theorem \cite{akhiezer1981theory}, for every $x \in \mathcal X$ and $y \in \mathcal Y$, there exists a mapping $\Phi:\mathcal{X}\times\mathcal{Y}\rightarrow \mathcal{H}$ such that $\Phi(x,y)$ is the representation of the linear functional mapping $f$ to $\langle y,f(x)\rangle_{\mathcal{Y}}$ in $\mathcal{H}$. Specifically, for all $f \in \mathcal H$,
\begin{equation}
\label{eq: reproducing}
\langle \Phi(x,y),f\rangle_{\mathcal{H}} = \langle y,f(x)\rangle_{\mathcal{Y}}.
\end{equation}
Given that the functional is linear in $y$, the representer can be written as $\Phi(x,y) = \Phi(x)y$, where $\Phi(x) \in \mathcal L(\mathcal Y, \mathcal H)$. The reproducing property then takes the form
\[
\langle \Phi(x)y,f\rangle_{\mathcal{H}} = \langle y,f(x)\rangle_{\mathcal{Y}}.
\]

Since $\Phi(x) y \in \mathcal H$ is a vector-valued operator, its evaluation at any point
$t \in \mathcal X$ yields an element in $\mathcal Y$. Accordingly, for every $x,t \in \mathcal X$, we introduce the nonlinear mapping $K:\mathcal{X}\times\mathcal{X}\rightarrow\mathcal{L}(\mathcal{Y})$ defined by
\[
K(x,t)y = (\Phi(x)y)(t).
\]

\begin{proposition}
\label{prop_1}
Given the nonlinear mapping $K$ defined above, the following statements hold:
\begin{enumerate}
    \item For all $y,z \in \mathcal{Y}$, and all $x,t \in \mathcal X$, $\langle y, K(x,t)z\rangle_{\mathcal{Y}} = \langle \Phi(x)y, \Phi(t)z\rangle_{\mathcal{H}}$\label{prop_1_1},\\
    \item For all $x,t \in \mathcal X$, $K(x,t)\in \mathcal{L}(\mathcal{Y})$, $K(x,x)\in \mathcal{L}_{+}(\mathcal{Y})$ and $K(x,t) = K(t,x)^*$,
    \item For any $m \in \mathbb N$, any $y_i \in \mathcal Y$, and any $x_i \in \mathcal X$ for $i \in [m]$, it holds that $\sum_{i = 1,j = 1}^m\innerp{\mathcal{Y}}{y_i}{K(x_i,x_j)y_j}\geq0$.\label{prop_1_3}
\end{enumerate}
\end{proposition}
Proofs of the above statements, as well as additional properties of $K$, can be found in \cite{micchelli2005learning}. Such a mapping is called an operator-valued kernel if it satisfies the properties stated in Proposition~\ref{prop_1}~\cite{micchelli2005learning}. Consequently, by the Moore--Aronszajn theorem~\cite{aronszajn1950theory,micchelli2005learning}, there exists a unique RKHS for which $K$ is the reproducing kernel.

% Here, we provide only the proof of \ref{prop_1_3}). 
% \begin{proof}Using property \ref{prop_1_1}), we have
% \begin{align*}
%     \sum_{i = 1,j = 1}^m&\innerp{\mathcal{Y}}{y_i}{K(x_i,x_j)y_j} \\
%     &= \sum_{i = 1,j = 1}^m\innerp{\mathcal{H}}{K_{x_i}y_i}{K_{x_j}y_j}\\
%     &=\Big\langle\sum_{i = 1}^mK_{x_i}y_i,\sum_{i = 1}^mK_{x_i}y_i\Big\rangle_{\mathcal{H}}\\
%     & = \Big\|\sum_{i = 1}^mK_{x_i}y_i\Big\|_{\mathcal{H}}^2\\
%     &\geq 0 \qedhere
% \end{align*}
% \end{proof}

\begin{remark}
    Unlike the scalar-valued case, the reproducing kernel of a vector-valued RKHS takes values in the space of bounded linear operators that act on the output space. When $\mathcal Y = \mathbb{R}^n$, the operator-valued kernel $K(x,t)$ can be represented by an $n \times n$ matrix. In the special case where $n = 1$, the kernel reduces to a scalar, and the reproducing property~\eqref{eq: reproducing} coincides with the classical reproducing property for scalar-valued functions. Finite-dimensional matrix-valued kernels are also closely related to multi-output Gaussian processes, where the dependence between two inputs is characterized by a covariance matrix~\cite{williams2006gaussian}. In this setting, each matrix entry encodes the correlation between different output coordinates.
\end{remark}
Given an operator-valued kernel $K$, the associated vector-valued RKHS is obtained as the closure of $\operatorname{span}\{ \Phi(x) y \mid x \in \mathcal X,\; y \in \mathcal Y \}$. As in the scalar-valued case, this space enjoys a universal approximation property. If the kernel $K$ is universal, then the corresponding RKHS is dense in the space of continuous operators mapping from $\mathcal X$ to $\mathcal Y$. We refer to~\cite{caponnetto2008universal,carmeli2010vector} for precise statements and conditions.

\subsection{Problem Formulation}
Consider the problem of learning an unknown nonlinear vector-valued operator
$f : \mathcal X \to \mathcal Y$. We assume that, for each input $x \in \mathcal X$, the corresponding measurement is corrupted by additive noise, i.e., $y = f(x) + \delta$, where $\delta \in \mathcal Y$ lies in the same output space. Given a set of pairwise distinct inputs $X: = \{x_1, \dots, x_m\}$, we observe the outputs $\boldsymbol{y}: = \{y_1, \dots, y_m\}$, where $y_i = f(x_i) + \delta_i$. This yields the dataset $\mathcal D: = \{(x_i, y_i)\}_{i=1}^m$. To make the learning problem properly defined, we first impose the following assumptions.

\begin{assumption}
\label{assumption_1}
    The noise is assumed to be bounded in norm by a known constant $\bar{\delta}$, i.e., for all $i\in[m], \|\delta_i\|_{\mathcal{Y}}\leq\bar{\delta}$.
\end{assumption}

By imposing an upper bound on the noise norm, the resulting analysis does not depend on any specific probabilistic assumptions on the noise distribution. This leads to a robust analytical framework that remains valid under arbitrary bounded disturbances.

\begin{assumption}
\label{assumption_2}
    The ground-truth operator $f$ is assumed to belong to the vector-valued RKHS $\mathcal H$ associated with a bounded kernel $K$. Moreover, its RKHS norm is bounded by a known constant $\Gamma$, i.e., $\|f\|_{\mathcal H} \le \Gamma$.
\end{assumption}
\begin{remark}
    A bounded norm of a vector-valued operator implies boundedness of its scalar projections.
More precisely, if $f \in \mathcal H$ has bounded norm, then for all $x \in \mathcal X$ and all $y \in \mathcal Y$,
the scalar quantities $|\langle f(x), y \rangle_{\mathcal Y}|$ are bounded, since 
\[
|\langle f(x), y \rangle_{\mathcal Y}| = |\langle f, \Phi(x)y \rangle_{\mathcal H}|\leq\|f\|_{\mathcal{H}}\|\Phi(x)\|\|y\|_{\mathcal{Y}}.
\]
It follows that $\|f(x)\|_{\mathcal Y}$ is bounded over $x \in \mathcal X$ since  \[
\|f(x)\|_{\mathcal Y}
= \sup_{y} \frac{|\langle f(x), y \rangle_{\mathcal Y}|}{\|y\|_{\mathcal{Y}}}\leq\|f\|_{\mathcal{H}}\|\Phi(x)\|
\]
This mirrors the scalar-valued RKHS case, where a bounded RKHS norm implies pointwise boundedness of the function values.
\end{remark}

\begin{assumption}
    The vector-valued RKHS $\mathcal{H}$ is normal, i.e., $\forall x\in \mathcal{X}, \nexists y \neq 0$ such that $\Phi(x)y = 0$.
\end{assumption}
\begin{remark}
    A vector-valued RKHS is said to be \emph{normal} if $\Phi(x) y = 0$ implies $y = 0$ for all $x \in \mathcal X$.
This property implies that the operator-valued kernel $K(x,x)$ is strictly positive definite. Indeed, for any $y \neq 0$, $\langle y, K(x,x) y \rangle_{\mathcal Y}
= \langle \Phi(x) y, \Phi(x) y \rangle_{\mathcal H}
= \|\Phi(x) y\|_{\mathcal H}^2 > 0$. Consequently, $K(x,x)$ is a positive and injective operator on $\mathcal Y$, and for any $z \in \mathcal Y$ the equation $K(x,x) y = z$ admits at most one solution. This injectivity property plays an important role in guaranteeing a unique solution in problems such as interpolation in vector-valued RKHS.
\end{remark}

The goal is to find $s^*\in \mathcal{H}$ that best explains the dataset $\mathcal{D}$. A common approach is to minimize the regularized risk functional, leading to the unconstrained optimization problem known as kernel ridge regression (KRR)
\begin{equation}
\label{eq_KRR}
    s^* = \arg\min_{s\in\mathcal{H}}\sum_{i = 1}^m\|s(x_i)-y_i\|_{\mathcal{Y}}^2 + \lambda \|s\|_{\mathcal{H}}^2,
\end{equation}
where $\lambda > 0$ is the regularization parameter. This parameter controls the smoothness of the learned model, i.e., larger values of $\lambda$ enforce stronger penalization of the RKHS norm and thus yield smoother, less complex solutions. Unlike the scalar-valued case, the classical scalar-valued representer theorem does not directly apply with the vector-valued output. Additional care is therefore required in the analysis of vector-valued KRR. In the following, we develop the solution to this problem and explicitly quantify the prediction error.

%% file: section/3_error_bounds.tex
\section{Vector-valued KRR}
\subsection{Interpolation}
To better understand the solution of vector-valued KRR, we start with the minimum-norm interpolation problem in the noiseless setting, where the observations satisfy $y_i = f(x_i)$. The goal is to find an interpolant $\bar{s} \in \mathcal H$ with minimal RKHS norm such that
$\bar{s}(x_i) = f(x_i)$ for all $i \in [m]$. This leads to the optimization problem
\begin{equation}
\label{eq_objective}
\begin{aligned}
    \bar{s} = &\arg\min_{s\in\mathcal{H}}\|s\|_{\mathcal{H}}^2,\\
    \text{s.t.} \quad &s(x_i) = f(x_i), \quad\forall i \in [m].
\end{aligned}
\end{equation}

In contrast to the scalar-valued case, $s(x_i) \in \mathcal Y$ are vector-valued outputs and therefore cannot be directly expressed in terms of linear evaluation functionals. Hence, the classical scalar-valued representer theorem~\cite{scholkopf2001generalized} does not apply. 

Alternative formulations and solution characterizations for this vector-valued interpolation problem have been developed in~\cite{micchelli2005learning,kadri2016operator}. 
The solution admits the following form
\begin{equation}
\label{representer}
    \bar{s}(x) = \sum_{i = 1}^mK(x,x_i)\bar{\alpha}_i,
\end{equation}
where the coefficients $\bar{\alpha}_i \in \mathcal Y$ are determined by the linear system equations
\begin{equation}
\label{eq: interpolation}
    \sum_{j = 1}K(x_i,x_j)\bar{\alpha}_j = f(x_i), \quad\forall i \in [m].
\end{equation}

In general, this system constitutes a linear equation over a possibly infinite-dimensional Hilbert space, and is therefore not directly tractable.
However, by restricting attention to specific classes of operator-valued kernels, additional analytical and computational tools become available.

Specifically, define the bounded linear operator
$K_{XX}:\mathcal{Y}^m \to \mathcal{Y}^m$ by
\[
(K_{XX}\bm{f})_i \;=\; \sum_{j=1}^m K(x_i,x_j)\,f(x_j),\quad \forall i \in [m].
\]
where $\bm{f}=(f(x_1),\dots,f(x_m))\in\mathcal{Y}^m$ and
\[
\mathcal{Y}^m := \underbrace{\mathcal{Y}\oplus\cdots\oplus\mathcal{Y}}_{m\ \text{copies}}
\]
is the direct sum of $m$ copies of $\mathcal{Y}$. The boundedness of $K_{XX}$ is straightforward as each $K(x_i,x_j)\in\mathcal{L}(\mathcal{Y})$ is bounded.
\begin{assumption}
\label{assumption_4}
  For each dataset $X$, the operator $K_{XX}$ is invertible. Equivalently, there exists a real number $c >0$ such that for all $\bm{y} \in \mathcal{Y}^m$, 
  \[
  \|K_{XX}\bm{y}\|_{\mathcal{Y}^m}\geq c\|\bm{y}\|_{\mathcal{Y}^m}.
  \]
\end{assumption}
\begin{remark}
    In the scalar-valued case, $K_{XX}$ corresponds to the kernel Gram matrix, which is invertible whenever the kernel is strictly positive definite and the input points are pairwise distinct. In contrast, for operator-valued kernels, invertibility of $K_{XX}$ is not guaranteed in general. Nevertheless, assuming invertibility is common in the literature~\cite{micchelli2005learning,kadri2016operator} and is typically satisfied in practice. In particular, invertibility holds for several important classes of operator-valued kernels, such as diagonal and separable kernels~\cite{kadri2016operator}.
\end{remark}
% \begin{remark}
%     Weaker assumptions, such as functional linear independence, can also be found in the literature \cite{micchelli2005learning}. This condition assumes that $K_{XX}$ is strictly positive, which guarantees uniqueness of the solution to \eqref{eq: interpolation}. However, this assumption does not guarantee that the inverse of $K_{XX}$ exists and is generally too weak to yield sufficiently rich analytical tools for deriving meaningful performance bounds.
% \end{remark}
Under Assumption~\ref{assumption_4}, the solution to~\eqref{eq: interpolation} can be written explicitly as
\[
\bar{\bm{\alpha}} = K_{XX}^{-1}\bm{f},
\]
where $\bar{\bm{\alpha}} = (\bar{\alpha}_1,\dots, \bar{\alpha}_m)\in \mathcal{Y}^m$.
For any $x \in \mathcal X$, we define the bounded linear operator
\[
K_{Xx}:\mathcal{Y}\rightarrow\mathcal{Y}^m
\]
by
\[
    \bigl(K_{Xx} y\bigr)_i := K(x_i,x) y, \quad \forall i \in [m].
\]
We denote its adjoint by $K_{xX} := K_{Xx}^*$. With this notation, the pointwise evaluation of $\bar{s}$ can be expressed as
\[
\bar{s}(x) = K_{xX}K_{XX}^{-1}\bm{f}.
\]

Consequently, we are now ready to analyze the prediction error between $\bar{s}(x)$ and $f(x)$. Analogous bounds have been established in the scalar-valued setting~\cite{maddalena2021deterministic}. However, additional care is required with the vector-valued outputs. As shown later, the bound derived here constitutes a direct generalization of the scalar-valued case.
\begin{proposition}
\label{prop_2}
    Let Assumptions~\ref{assumption_2}-\ref{assumption_4} hold. For any $x\in\mathcal{X}$, the distance between $\bar{s}(x)$ and $f(x)$ in the $\mathcal{Y}$-norm is bounded by
    \[
    \|\bar{s}(x)-f(x)\|_{\mathcal{Y}}\leq \|P^{\frac{1}{2}}(x)\|\sqrt{\Gamma^2-\|\bar{s}\|_{\mathcal{H}}^2},
    \]
    where $P^{\frac{1}{2}}(x)$ is the positive square root of $P(x) = K(x,x)-K_{xX}K_{XX}^{-1}K_{Xx}$, $\|\bar{s}\|_{\mathcal{H}}^2 = \langle \bm{f},K_{XX}^{-1}\bm{f}\rangle_{\mathcal{Y}^m}$.
\end{proposition}
\begin{proof}
    Given any fixed $x \in \mathcal X$ and the corresponding output $f(x)$, we denote by $\bar{s}_{+}$ the minimum-norm interpolant of the augmented dataset $\mathcal D \cup \{(x, f(x))\}$. That is, $\bar{s}_{+} \in \mathcal H$ interpolates both the original dataset and the additional point $(x,f(x))$. 
    The RKHS norm of $\bar{s}_{+}$ can then be computed explicitly by
    % \begin{align*}
    %     &\|\bar{s}_{+}\|_{\mathcal{H}}^2\\
    %     &= \bigg\langle \begin{pmatrix}
    %         \bm{f},\\ f(x)
    %     \end{pmatrix},
    %     \begin{bmatrix}
    %         K_{XX},&K_{Xx}\\
    %         K_{xX},&K_{xx}
    %     \end{bmatrix}^{-1}
    %     \begin{pmatrix}
    %         \bm{f},\\ f(x)
    %     \end{pmatrix}\bigg \rangle_{\mathcal{Y}^{m+1}}\\
    %     &\overset{(1)}{=}\bigg\langle \begin{pmatrix}
    %         \bm{f},\\ f(x)
    %     \end{pmatrix},
    %     \begin{bmatrix}
    %         K_{XX}^{-1},&0\\
    %         0,&0
    %     \end{bmatrix}
    %     \begin{pmatrix}
    %         \bm{f},\\ f(x)
    %     \end{pmatrix}\bigg \rangle_{\mathcal{Y}^{m+1}}\\
    %     &+\bigg\langle \!\begin{pmatrix}
    %         \bm{f},\\ f(x)
    %     \end{pmatrix}\!,\!
    %     \begin{bmatrix}
    %         K_{XX}^{-1}\!K_{Xx}\!P^{-\frac{1}{2}}(x)\\
    %         -P^{-\frac{1}{2}}(x)
    %     \end{bmatrix}
    %     \begin{bmatrix}
    %         K_{XX}^{-1}\!K_{Xx}\!P^{-\frac{1}{2}}(x)\\
    %         -P^{-\frac{1}{2}}(x)
    %     \end{bmatrix}^*\!
    %     \begin{pmatrix}
    %         \bm{f},\\ f(x)
    %     \end{pmatrix}\!\bigg \rangle_{\mathcal{Y}^{m\!+\!1}}\\
    %     &=\|\bar{s}\|_{\mathcal{H}}^2 + \|P^{-\frac{1}{2}}(x)\big(\bar{s}(x)-f(x)\big)\|_{\mathcal{Y}}^2,
    % \end{align*}
    \begin{align*}
        \|\bar{s}&_{+}\|_{\mathcal{H}}^2\\
        &= \bigg\langle \begin{bmatrix}
            \bm{f}\\ f(x)
        \end{bmatrix},
        \begin{bmatrix}
            K_{XX},&K_{Xx}\\
            K_{xX},&K_{xx}
        \end{bmatrix}^{-1}
        \begin{bmatrix}
            \bm{f}\\ f(x)
        \end{bmatrix}\bigg \rangle_{\mathcal{Y}^{m+1}}\\
        &\overset{(1)}{=}\bigg\langle \begin{bmatrix}
            \bm{f}\\ f(x)
        \end{bmatrix},
        \begin{bmatrix}
            K_{XX}^{-1},&0\\
            0,&0
        \end{bmatrix}
        \begin{bmatrix}
            \bm{f}\\ f(x)
        \end{bmatrix}\bigg \rangle_{\mathcal{Y}^{m+1}}\\
        &\quad+\bigg\|
        \begin{bmatrix}
            K_{XX}^{-1}\!K_{Xx}\!P^{-\frac{1}{2}}(x)\\
            -P^{-\frac{1}{2}}(x)
        \end{bmatrix}^*\!
        \begin{bmatrix}
            \bm{f}\\ f(x)
        \end{bmatrix}\!\bigg \|_{\mathcal{Y}}^2\\
        &=\|\bar{s}\|_{\mathcal{H}}^2 + \|P^{-\frac{1}{2}}(x)\big(\bar{s}(x)-f(x)\big)\|_{\mathcal{Y}}^2,
    \end{align*}
    where in (1), we apply the Schur complement for the block operator matrix. Under Assumption~\ref{assumption_4}, the operator $P(x)$ is bounded below and therefore invertible. As a consequence, the square root operator $P^{\frac{1}{2}}(x)$ exists and is also invertible. 
    
    Since $\|\bar{s}_{+}\|_{\mathcal{H}}^2\leq\|f\|_{\mathcal{H}}^2\leq\Gamma^2$, we have
    \begin{align*}
        \|\bar{s}(x)-f(x)\|_{\mathcal{Y}} &= \|P^{\frac{1}{2}}(x)P^{-\frac{1}{2}}(x)\big(\bar{s}(x)-f(x)\big)\|_{\mathcal{Y}}\\
        &\leq\|P^{\frac{1}{2}}(x)\|\|P^{-\frac{1}{2}}(x)\big(\bar{s}(x)-f(x)\big)\|_{\mathcal{Y}}\\
        &\leq\|P^{\frac{1}{2}}(x)\|\sqrt{\Gamma^2-\|\bar{s}\|_{\mathcal{H}}^2}\qedhere
    \end{align*}
\end{proof}
From Proposition~\ref{prop_2}, we obtain a bound on the prediction error measured in the $\mathcal Y$-norm. When $\mathcal Y = \mathbb{R}$, we recover the same bound as in~\cite[Proposition~1]{maddalena2021deterministic}. The operator $P(x)$ reduces to a scalar, which coincides with the posterior variance at $x$ in the Gaussian process framework. When $\mathcal{Y}$ is a function space, such a bound does not, in general, imply a pointwise bound on the difference between function values. For instance, this arises when $\mathcal{Y}$ is the Hilbert space of square-integrable functions. However, in more structured output spaces, such as when $\mathcal{Y}$ is a scalar-valued RKHS, boundedness in norm implies uniform boundedness.

\subsection{Kernel Ridge Regression}
As in the interpolation problem~\eqref{eq_objective}, the classical scalar-valued representer theorem does not directly apply for problem \eqref{eq_KRR}. Alternative derivations for the vector-valued setting are provided in~\cite{kadri2016operator,micchelli2005learning}. The solution admits the same representer form as in~\eqref{representer}, where $\alpha^*$ is determined by the linear system
\[
(K_{XX} + \lambda I)\bm{\alpha}^* = \bm{y},
\]
where we denote $(y_1,\dots,y_m)\in\mathcal{Y}^m$ and $(\alpha_1^*,\dots,\alpha_m^*)\in\mathcal{Y}^m$ by $\bm{y}$ and $\bm{\alpha}^*$, respectively. We denote the identity operator by $I:\mathcal{Y}^m\rightarrow\mathcal{Y}^m$. 

The operator $K_{XX}+\lambda I$ is directly invertible as $K_{XX}\in\mathcal{L}_{+}(\mathcal{Y}^m)$. Consequently, the solution of \eqref{eq_KRR} admits the explicit form
\begin{equation}
\label{eq_KRR_sol}
    s^*(x)=K_{xX}\bm{\alpha}^*,
\end{equation}
where $\bm{\alpha}^*=(K_{XX}+\lambda I)^{-1}\boldsymbol{y}$.

With this formulation, an upper bound on the prediction error
$\|s^*(x) - f(x)\|_{\mathcal Y}$ for any $x \in \mathcal X$ can be derived,
as summarized in the following theorem.
\begin{theorem}
\label{theorem_1}
    Let Assumptions~\ref{assumption_1}-\ref{assumption_4} hold. For any $x \in \mathcal X$, the error between the vector-valued KRR model and the ground-truth output satisfies
    \begin{align*}
        \|s^*(x)-f(x)\|_{\mathcal{Y}}&\leq \|P^{\frac{1}{2}}(x)\|\sqrt{\Gamma^2 + \Delta -\|\hat{s}\|_{\mathcal{H}}^2}+\nabla\\
        &\quad+\Big\|K_{xX}\big((K_{XX}+\lambda I)^{-1}-K_{XX}^{-1}\big)\bm{y}\Big\|_{\mathcal{Y}},
    \end{align*}
    where $\hat{s}$ denotes the minimum-norm interpolant of the noisy measurements, i.e. $\hat{s}(x) = K_{xX}K_{XX}^{-1}\bm{y}$, $\Delta$ is the optimal value of 
    \begin{equation}
        \label{eq_delta_1}
    \begin{aligned}
        \max_{\boldsymbol{\delta}\in \mathcal Y^m}\quad&2\langle \bm{\delta}, K_{XX}^{-1}\bm{y}\rangle_{\mathcal{Y}^m} -  \langle \bm{\delta}, K_{XX}^{-1}\bm{\delta}\rangle_{\mathcal{Y}^m},\\
        \text{s.t.}\quad&\|\delta_i\|_{\mathcal{Y}}\leq \bar{\delta},\quad \forall i \in [m],
    \end{aligned}
    \end{equation}
    and $\nabla$ is the optimal value of
    \begin{equation}
    \label{eq_delta_2}
        \begin{aligned}
            \max_{\boldsymbol{\delta}\in \mathcal Y^m}\quad&\|K_{xX}K_{XX}^{-1}\bm{\delta}\|_{\mathcal{Y}},\\
            \text{s.t.}\quad&\|\delta_i\|_{\mathcal{Y}}\leq \bar{\delta},\quad \forall i \in [m],
        \end{aligned}
    \end{equation}
    with $\bm{\delta} = (\delta_1,\dots,\delta_m)\in \mathcal{Y}^m$.
\end{theorem}
Before presenting the proofs, we observe that the bound consists of three terms.
The first term has the same structure as in Proposition~\ref{prop_2}.
In fact, a tighter bound replaces
\[
\|P^{\frac{1}{2}}(x)\|\sqrt{\Gamma^2+\Delta-\|\hat{s}\|_{\mathcal H}^2}
\]
by
\[
\|P^{\frac{1}{2}}(x)\|\sqrt{\Gamma^2-\|\bar{s}\|_{\mathcal H}^2}.
\]
However, evaluating $\|\bar{s}\|_{\mathcal H}$ requires access to the ground-truth values $\boldsymbol{f}$, as indicated in Proposition~\ref{prop_2}. Hence, despite yielding a sharper estimate, this tighter expression is not computable in practice. We therefore introduce a tractable relaxation, in which the unknown quantity $\Gamma^2-\|\bar{s}\|_{\mathcal H}^2$ is upper bounded by $\Gamma^2+\Delta-\|\hat{s}\|_{\mathcal H}^2$.
The remaining two terms quantify the residual contribution due to measurement noise.

\begin{remark}[Sensitivity to $\Gamma$ and $\bar{\delta}$]
The error bound depends explicitly on the RKHS norm bound $\Gamma$ and the noise bounds $\bar{\delta}$, which in practice can be estimated conservatively from historical data. For fixed $\Delta$ and $\hat{s}$, the dependence on $\Gamma$ is monotone increasing in $\Gamma$ and becomes asymptotically linear as $\Gamma$ increases. The quantities $\Delta$ and $\nabla$ are also monotone nondecreasing with respect to $\bar{\delta}$, since increasing the noise bounds enlarges the feasible sets of~\eqref{eq_delta_1} and~\eqref{eq_delta_2}. Consequently, conservative estimates of either $\Gamma$ or $\bar{\delta}$ lead to more conservative prediction bounds.
\end{remark}

% \begin{remark}
%     In the scalar-kernel setting, the objective in~\eqref{eq_delta_2} reduces to$\left| \langle \boldsymbol{\delta}, K_{XX}^{-1} K_{Xx} \rangle_{\mathbb{R}^m} \right|$. In this case, the optimization problem admits an explicit solution given by $\left\langle \bar{\boldsymbol{\delta}}, \, \bigl| K_{XX}^{-1} K_{Xx} \bigr| \right\rangle_{\mathbb{R}^m}$, where $\bar{\boldsymbol{\delta}} = [\bar{\delta},\dots,\bar{\delta}]^{\top}$. This coincides with the formulation presented in~\cite{maddalena2021deterministic}. In contrast, when a vector-valued kernel is employed, the problem no longer admits a closed-form solution due to the additional structure induced by the output Hilbert space.
% \end{remark}

\begin{proof}
    We first express $s^*(x)$ in terms of the interpolation model $\bar{s}(x)$. Recalling that $\boldsymbol{y}=\boldsymbol{f}+\boldsymbol{\delta}$, we have
    \begin{align*}
        s^*(x) &= K_{xX}K_{XX}^{-1}\bm{y}\\
        &\quad+ K_{xX}(K_{XX}+\lambda I)^{-1}\bm{y}\!-\!K_{xX}K_{XX}^{-1}\bm{y}\\
        &= \bar{s}(x) + K_{xX}K_{XX}^{-1}\bm{\delta}\\
        &\quad+ K_{xX}\big((K_{XX}+\lambda I)^{-1}-K_{XX}^{-1}\big)\bm{y}.
    \end{align*}
    Consequently, by the triangle inequality,
    \begin{align*}
        \|s^*(x)&-f(x)\|_{\mathcal{Y}}\\
        &\leq\|\bar{s}(x) - f(x)\|_{\mathcal{Y}}+\nabla\\
        &\quad+\|K_{xX}\big((K_{XX}+\lambda I)^{-1}-K_{XX}^{-1}\big)\bm{y}\|_{\mathcal{Y}}\\
        &\overset{\text{(1)}}{\leq}\|P^{\frac{1}{2}}(x)\|\sqrt{\Gamma^2-\|\bar{s}\|_{\mathcal H}^2}+\nabla\\
        &\quad+\|K_{xX}\big((K_{XX}+\lambda I)^{-1}-K_{XX}^{-1}\big)\bm{y}\|_{\mathcal{Y}}\\
        &\overset{\text{(2)}}{\leq}\|P^{\frac{1}{2}}(x)\|\sqrt{\Gamma^2+\Delta-\|\hat{s}\|_{\mathcal H}^2}+\nabla\\
        &\quad+\|K_{xX}\big((K_{XX}+\lambda I)^{-1}-K_{XX}^{-1}\big)\bm{y}\|_{\mathcal{Y}},
    \end{align*}
    where step (1) follows from Proposition~\ref{prop_2}. In step (2), the result is obtained by expanding both $\|\bar{s}\|_{\mathcal{H}}^2$ and $\|\hat{s}\|_{\mathcal{H}}^2$ with the inner product expression and the fact that $\Delta$ is the optimal value of problem \eqref{eq_delta_1}.
\end{proof}

% \subsection{SVM (optional)}

%% file: section/5_numerical_result.tex
\section{Numerical results}
To validate the theoretical results, we conduct numerical experiments under two scenarios. 
In the first setting, we consider learning an operator whose output is a correlated two-dimensional vector. 
In the second setting, we consider learning an unknown operator from the input space to a scalar-valued RKHS.

In both cases, the input space is chosen as $\mathcal X = [0,1] \subset \mathbb{R}$, and the scalar-valued kernel $G$ is the Gaussian kernel
\begin{equation}
\label{eq: G}
    G(s,t) = \exp\!\left(-\frac{(s-t)^2}{2\ell^2}\right).
\end{equation}

The length scale $\ell$ is set to 0.1.

\subsection{Two-Dimensional Vector Output}

\begin{figure}[t]
\centering
\includegraphics[width=0.75\linewidth]{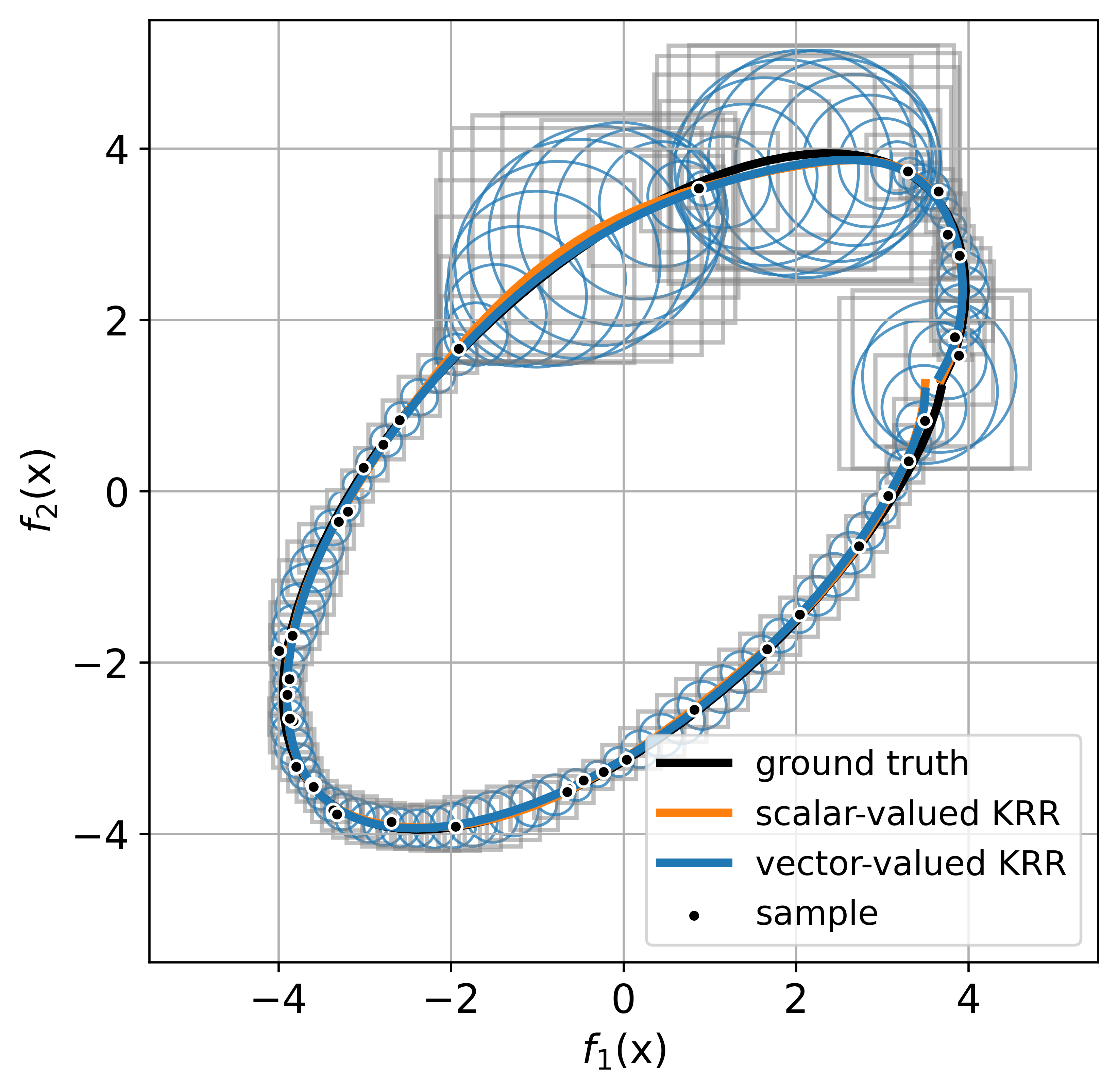}
\caption{Vector-valued KRR model compared with the ground truth in 2D. The input $x$ is the polar angle. The blue circles represent the vector-valued KRR bound, and the grey boxes represent the scalar-valued KRR bound.}
\label{fig ellipse}
\end{figure}

\begin{figure}[t]
\centering
\includegraphics[width=\linewidth]{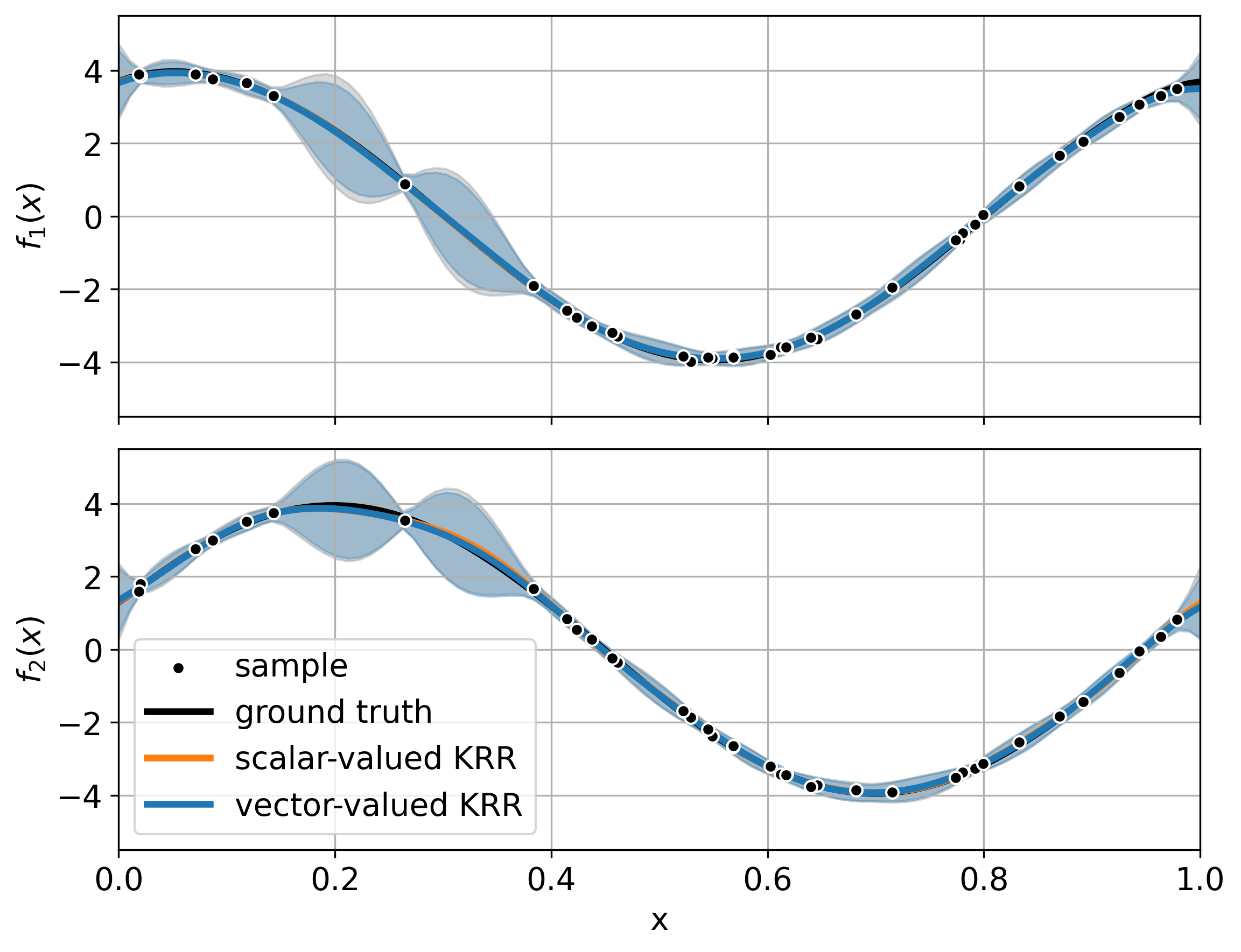}
\caption{Vector-valued KRR model compared with the ground truth in each coordinate.}
\label{fig separate}
\end{figure}
In this experiment, the ground-truth operator is defined as the weighted sum of kernels,
\[
f(x) = \sum_{i=1}^{50} G(x,x_i)\,T\,\alpha_i,
\]
where
\[
T = \Phi \Lambda \Phi^\top,
\qquad
\Lambda =
\begin{bmatrix}
1 & 0 \\
0 & 0.5
\end{bmatrix},
\]
and $\Phi$ is the rotation matrix that rotates any vector in $\mathbb{R}^2$ counterclockwise around the origin by $45^\circ$.

The coefficients $\alpha_i \in \mathbb{R}^2$ are obtained by solving the regularized least-squares problem
\begin{equation}
\label{eq_KRR_sim}
\min_{\{\alpha_i\}}
\sum_{i=1}^{50}
\|f(x_i) - y_i\|_{\mathbb{R}^2}^2
+ 0.01\,\|f\|_{\mathcal H}^2,
\end{equation}
where $x_i$ is generated by uniformly sampling from $[0,1]$ and $y_i$ are sampled from points on an ellipse
\[
y_i = T\begin{bmatrix}
    5\cos{2\pi x_i}\\
    5\sin{2\pi x_i}
\end{bmatrix}.
\]

Consequently, the operator $f$ provides a kernel representation of the ellipse, with RKHS norm $\|f\|_{\mathcal H} = 10.55$. Training data are generated by sampling from $f$ at 40 randomly selected input locations with i.i.d. noise samples $\delta_i \in \mathbb{R}^2$ satisfying $\|\delta_i\|_2 \le 0.1$. We train a vector-valued KRR model $s^*$ as formulated in~\eqref{eq_KRR}. 
For comparison, we also train two independent scalar-valued models, one for each output coordinate, using the same dataset. The learned model and the uncertainty bound are illustrated in Fig. \ref{fig ellipse}.

In Fig.~\ref{fig ellipse}, $f$ is shown as the black solid curve, while $s^*$ is plotted as the blue curve. The model obtained by training independent scalar-valued KRR models for each output dimension is shown in orange.
The black dots represent the training data. Both the vector-valued and scalar-valued KRR models produce very similar mean predictions. This is expected since they are trained on the same dataset. However, these two approaches produce fundamentally different uncertainty bounds. For each input, the vector-valued KRR model provides an upper bound on the prediction error in terms of the $2$-norm. This bound is illustrated by the blue circle in Fig.~\ref{fig ellipse}. In contrast, the scalar-valued models provide independent bounds for each coordinate.
These bounds form a box constraint in the output space, shown as the grey box, which is generally more conservative.

The model predictions and their corresponding bounds projected onto each output dimension are shown in Fig.~\ref{fig separate}. Although the scalar-valued bounds appear slightly more conservative than those derived from the vector-valued model, the difference is not substantial in this example. This is because the bounds are obtained by projecting the norm bound onto each coordinate, which represents a worst-case scenario. For example, at $x=0.2$, the bounds for both $f_1$ and $f_2$ appear large. However, they cannot simultaneously attain their maximal values without violating the overall norm constraint. A large value in one coordinate necessarily implies a smaller value in the other. This coupling between outputs is captured by the vector-valued KRR model, whereas the scalar-valued KRR model fails to account for such dependencies.

\subsection{Scalar-Valued RKHS Output}

% \begin{figure}[t]
% \centering
% \includegraphics[width=0.7\linewidth]{picture/3D.png}
% \caption{Vector-valued KRR with function output.}
% \label{fig 3d}
% \end{figure}
% \begin{figure}[t]
%     \subfloat[Model prediction.\label{fig 3d model}]{
%        \includegraphics[width=0.48\linewidth]{picture/3D_model.png}}
%     \hfill
%     \subfloat[Error.\label{fig 3d error}]{
%         \includegraphics[width=0.48\linewidth]{picture/3D_error.png}}
%     \caption{Vector-valued KRR model with function output.}
%     \label{fig 3D main}
% \end{figure}

\begin{figure}[t]
\centering
\includegraphics[width=\linewidth]{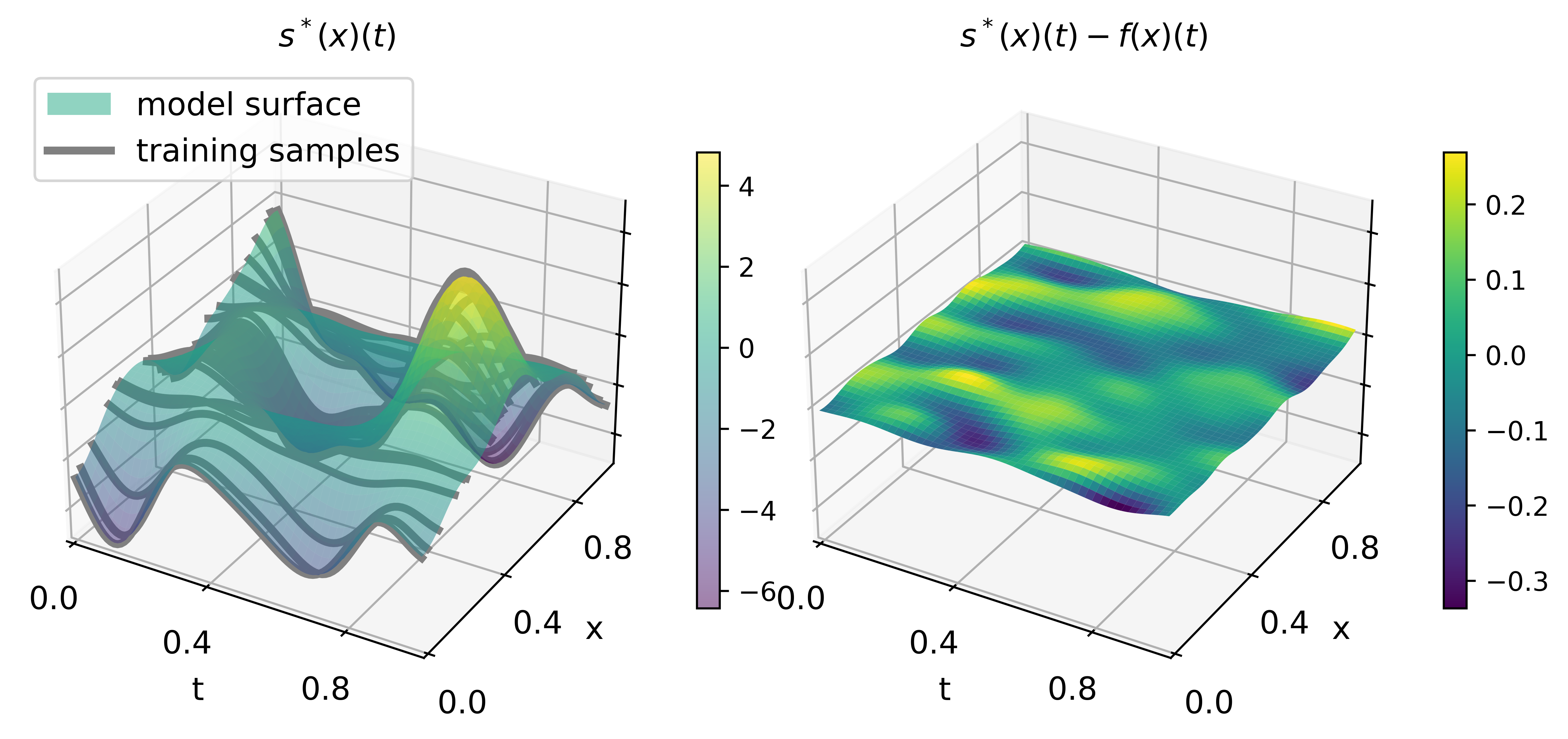}
\caption{Vector-valued KRR model with function output.}
\label{fig 3D main}
\end{figure}

\begin{figure}[t]
\centering
\includegraphics[width=\linewidth]{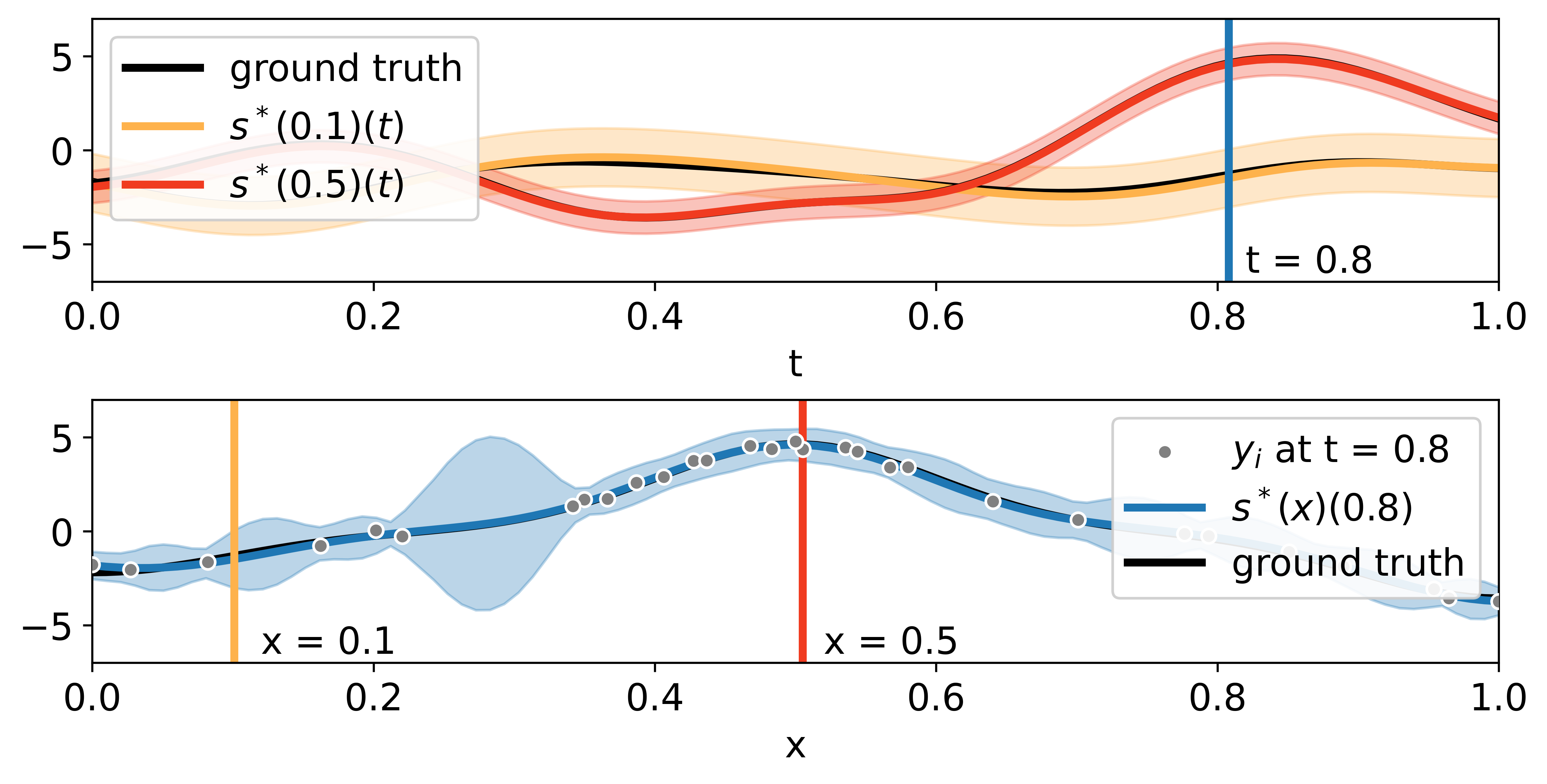}
\caption{Vector-valued KRR model and error bound in each coordinate.}
\label{fig tx}
\end{figure}

We consider the problem of learning a nonlinear operator where the outputs are scalar-valued functions. We assume that the output functions belong to the scalar-valued RKHS induced by \eqref{eq: G}.

The ground-truth operator is defined as
\[
f(x) = \sum_{i=1}^{N} G(x,x_i)\,\alpha_i,
\]
where, in the case, $T$ is the identity operator and $\alpha_i$ lies in the scalar-valued RKHS. We represent $\alpha_i(t) = \sum_{j=1}^{9} G(t,t_j)\,\beta_{ij}$, where $\{t_j\}_{j=1}^{9}$ are points sampled from a uniform grid on $[0,1]$, and the coefficients $\beta_{ij}$ are generated independently from the uniform distribution on $[-5,5]$. The resulting RKHS norm of the ground-truth operator is $\|f\|_{\mathcal H}=13.49$.

We randomly collect $30$ inputs $\{x_i\}_{i=1}^{30}$ and sample $y_i = f(x_i) + \delta_i$ from the ground-truth operator with additive noise $\delta_i$ satisfying $\|\delta_i\|_{\mathcal H}\le 0.5$ for all $i\in[30]$. Here, we restrict $\delta$ and $\alpha$ to be based on the fixed kernel centers $\{t_j\}_{j=1}^9$. However, in general, they can be represented using their own set of kernel centers. The vector-valued KRR model is then trained with $\lambda=0.01$.

The learned model $s^*$ and the corresponding samples are illustrated in the left plot in Fig.~\ref{fig 3D main}. Here, $s^*$ is a surface, while the noisy function samples are shown as grey curves along the $t$ axis. We conclude that $s^*$ successfully learns the ground truth from the function samples and provides a smooth interpolation over the entire input domain. This is further demonstrated by the prediction error plotted on the right plot in Fig.~\ref{fig 3D main}.

To illustrate the norm bound on the prediction error, we fix one coordinate and plot the model prediction along both the $t$-axis and the $x$-axis in Fig.~\ref{fig tx}. We denote the function-valued output by $s^*(x)\in \mathcal{Y}$ and its pointwise evaluation by $s^*(x)(t)\in\mathbb{R}$. In the upper plot, the function values $s^*(0.1)(t)$ and $s^*(0.5)(t)$ are shown for $t \in [0,1]$. The ground-truth output, indicated by the black solid line, lies within the derived error bounds (shaded region). It can be observed that the bound is uniform with respect to $t$. This is because the measurements are function-valued and the derived bound is expressed in terms of the scalar-valued RKHS norm. In such spaces, a bounded norm implies that the function values are uniformly bounded.

In the lower plot, the function values $s^*(x)(0.8)$ are shown for $x \in [0,1]$. The pointwise evaluations of $y_i$ at $t = 0.8$ are marked by the grey dots. The error bound is shown as the transparent area. As expected, it becomes tighter in regions where the samples are dense and more conservative in regions where the samples are sparse. Similar to the two-dimensional case, for each input $x$, the bound can be interpreted as the projection of a high-dimensional norm ball onto a one-dimensional space, which represents a worst-case scenario. In particular, a larger deviation at $t = 0.8$ implies a smaller deviation at other locations $t \in [0,1]\setminus\{0.8\}$. 

Overall, this framework provides a principled way to handle general vector-valued measurements, a setting that cannot be addressed using existing scalar-valued learning frameworks.

%% file: section/6_conclusion.tex
\section{Conclusion}
In this work, we study a generalized nonlinear operator learning problem in which the outputs lie in arbitrary Hilbert spaces, and formulate it within the framework of vector-valued KRR. By assuming that the unknown operator belongs to a vector-valued RKHS and that the measurement noise has bounded norm, we derive a deterministic bound on the discrepancy between the vector-valued KRR estimator and the ground truth. The theoretical results are further validated through simulation studies under two different scenarios. The proposed approach provides a framework for robust learning of nonlinear operators with generalized vector-valued outputs. Future work includes extending the framework to consider more complex observation models, such as measurements obtained through linear operators, and applying the derived bounds to controller design problems, such as robust MPC.

% \red{add what kind of problem we try to solve}

% \red{adjust the logic of error bound and efficient computation logic}

% \red{Which bound is tighter? $P_{\lambda}(x)\Gamma + K_{xX}(K_{XX}+\lambda I)^{-1}\bm{\delta}$}

%% file: main.bbl
\begin{thebibliography}{10}
\providecommand{\url}[1]{#1}
\csname url@samestyle\endcsname
\providecommand{\newblock}{\relax}
\providecommand{\bibinfo}[2]{#2}
\providecommand{\BIBentrySTDinterwordspacing}{\spaceskip=0pt\relax}
\providecommand{\BIBentryALTinterwordstretchfactor}{4}
\providecommand{\BIBentryALTinterwordspacing}{\spaceskip=\fontdimen2\font plus
\BIBentryALTinterwordstretchfactor\fontdimen3\font minus \fontdimen4\font\relax}
\providecommand{\BIBforeignlanguage}[2]{{%
\expandafter\ifx\csname l@#1\endcsname\relax
\typeout{** WARNING: IEEEtran.bst: No hyphenation pattern has been}%
\typeout{** loaded for the language `#1'. Using the pattern for}%
\typeout{** the default language instead.}%
\else
\language=\csname l@#1\endcsname
\fi
#2}}
\providecommand{\BIBdecl}{\relax}
\BIBdecl

\bibitem{pillonetto2014kernel}
G.~Pillonetto, F.~Dinuzzo, T.~Chen, G.~De~Nicolao, and L.~Ljung, ``Kernel methods in system identification, machine learning and function estimation: A survey,'' \emph{Automatica}, vol.~50, no.~3, pp. 657--682, 2014.

\bibitem{10266828}
A.~Carè, R.~Carli, A.~D. Libera, D.~Romeres, and G.~Pillonetto, ``Kernel methods and {G}aussian processes for system identification and control: A road map on regularized kernel-based learning for control,'' \emph{IEEE Control Systems Magazine}, vol.~43, no.~5, pp. 69--110, 2023.

\bibitem{pillonetto2024kernel}
G.~Pillonetto and M.~Bisiacco, ``Kernel-based linear system identification: When does the representer theorem hold?'' \emph{Automatica}, vol. 159, p. 111347, 2024.

\bibitem{dinuzzo2015kernels}
F.~Dinuzzo, ``Kernels for linear time invariant system identification,'' \emph{SIAM Journal on Control and Optimization}, vol.~53, no.~5, pp. 3299--3317, 2015.

\bibitem{lian2021nonlinear}
Y.~Lian and C.~N. Jones, ``Nonlinear data-enabled prediction and control,'' in \emph{Learning for Dynamics and Control}.\hskip 1em plus 0.5em minus 0.4em\relax PMLR, 2021, pp. 523--534.

\bibitem{lian2020gaussian}
L.~Yingzhao and C.~Jones, ``On {G}aussian process based {K}oopman operators,'' vol.~53, no.~2.\hskip 1em plus 0.5em minus 0.4em\relax IFAC-Papers{O}nline, 2020, p. 52–58.

\bibitem{scampicchio2025gaussian}
A.~Scampicchio, E.~Arcari, A.~Lahr, and M.~N. Zeilinger, ``Gaussian processes for dynamics learning in model predictive control,'' \emph{Annual Reviews in Control}, vol.~60, p. 101034, 2025.

\bibitem{7330913}
F.~Berkenkamp and A.~P. Schoellig, ``Safe and robust learning control with {G}aussian processes,'' in \emph{European Control Conference}, 2015, pp. 2496--2501.

\bibitem{zhao2015distributed}
J.~Zhao and F.~D{\"o}rfler, ``Distributed control and optimization in {DC} microgrids,'' \emph{Automatica}, vol.~61, pp. 18--26, 2015.

\bibitem{nedic2018network}
A.~Nedi{\'c}, A.~Olshevsky, and M.~G. Rabbat, ``Network topology and communication-computation tradeoffs in decentralized optimization,'' \emph{Proceedings of the IEEE}, vol. 106, no.~5, pp. 953--976, 2018.

\bibitem{li2020fourier}
Z.~Li, N.~Kovachki, K.~Azizzadenesheli, B.~Liu, K.~Bhattacharya, A.~Stuart, and A.~Anandkumar, ``Fourier neural operator for parametric partial differential equations,'' \emph{arXiv preprint arXiv:2010.08895}, 2020.

\bibitem{gertheiss2024functional}
J.~Gertheiss, D.~R{\"u}gamer, B.~X. Liew, and S.~Greven, ``Functional data analysis: An introduction and recent developments,'' \emph{Biometrical Journal}, vol.~66, no.~7, p. e202300363, 2024.

\bibitem{kadri2010nonlinear}
H.~Kadri, E.~Duflos, P.~Preux, S.~Canu, and M.~Davy, ``Nonlinear functional regression: {A} functional {RKHS} approach,'' in \emph{Proceedings of the 13th International Conference on Artificial Intelligence and Statistics}.\hskip 1em plus 0.5em minus 0.4em\relax JMLR Workshop and Conference Proceedings, 2010, pp. 374--380.

\bibitem{yuan2010reproducing}
M.~Yuan and T.~T. Cai, ``A reproducing kernel {H}ilbert space approach to functional linear regression,'' \emph{The Annals of Statistics}, vol.~38, no.~6, 2010.

\bibitem{chen2016variable}
Y.~Chen, J.~Goldsmith, and R.~T. Ogden, ``Variable selection in function-on-scalar regression,'' \emph{Stat}, vol.~5, no.~1, pp. 88--101, 2016.

\bibitem{bevanda2023koopman}
P.~Bevanda, M.~Beier, A.~Lederer, S.~Sosnowski, E.~H{\"u}llermeier, and S.~Hirche, ``Koopman kernel regression,'' \emph{Advances in Neural Information Processing Systems}, vol.~36, pp. 16\,207--16\,221, 2023.

\bibitem{ruiz2024survey}
C.~Ruiz, C.~M. Ala{\'\i}z, and J.~R. Dorronsoro, ``A survey on kernel-based multi-task learning,'' \emph{Neurocomputing}, vol. 577, p. 127255, 2024.

\bibitem{caponnetto2007optimal}
A.~Caponnetto and E.~De~Vito, ``Optimal rates for the regularized least-squares algorithm,'' \emph{Foundations of Computational Mathematics}, vol.~7, no.~3, pp. 331--368, 2007.

\bibitem{audiffren2013stability}
J.~Audiffren and H.~Kadri, ``Stability of multi-task kernel regression algorithms,'' in \emph{Asian Conference on Machine Learning}.\hskip 1em plus 0.5em minus 0.4em\relax PMLR, 2013, pp. 1--16.

\bibitem{park2017remarks}
C.~Park, R.~T. Haftka, and N.~H. Kim, ``Remarks on multi-fidelity surrogates,'' \emph{Structural and Multidisciplinary Optimization}, vol.~55, no.~3, pp. 1029--1050, 2017.

\bibitem{pan2009survey}
S.~J. Pan and Q.~Yang, ``A survey on transfer learning,'' \emph{IEEE Transactions on Knowledge and Data Engineering}, vol.~22, no.~10, pp. 1345--1359, 2009.

\bibitem{liu2018remarks}
H.~Liu, J.~Cai, and Y.-S. Ong, ``Remarks on multi-output {Gaussian} process regression,'' \emph{Knowledge-Based Systems}, vol. 144, pp. 102--121, 2018.

\bibitem{kadri2009general}
H.~Kadri, E.~Duflos, M.~Davy, P.~Preux, and S.~Canu, ``General framework for nonlinear functional regression with reproducing kernel {H}ilbert spaces,'' Ph.D. dissertation, INRIA, 2009.

\bibitem{micchelli2005learning}
C.~A. Micchelli and M.~Pontil, ``On learning vector-valued functions,'' \emph{Neural Computation}, vol.~17, no.~1, pp. 177--204, 2005.

\bibitem{akhiezer1981theory}
N.~I. Akhiezer and I.~M. Glazman, \emph{Theory of linear operators in Hilbert space}.\hskip 1em plus 0.5em minus 0.4em\relax Pitman Publishing, 1981, no.~9.

\bibitem{aronszajn1950theory}
N.~Aronszajn, ``Theory of reproducing kernels,'' \emph{Transactions of the American Mathematical Society}, vol.~68, no.~3, pp. 337--404, 1950.

\bibitem{williams2006gaussian}
C.~K. Williams and C.~E. Rasmussen, \emph{Gaussian processes for machine learning}.\hskip 1em plus 0.5em minus 0.4em\relax MIT Press, Cambridge, MA, 2006, vol.~2, no.~3.

\bibitem{caponnetto2008universal}
A.~Caponnetto, C.~A. Micchelli, M.~Pontil, and Y.~Ying, ``Universal multi-task kernels,'' \emph{The Journal of Machine Learning Research}, vol.~9, pp. 1615--1646, 2008.

\bibitem{carmeli2010vector}
C.~Carmeli, E.~De~Vito, A.~Toigo, and V.~Umanit{\'a}, ``Vector valued reproducing kernel {H}ilbert spaces and universality,'' \emph{Analysis and Applications}, vol.~8, no.~01, pp. 19--61, 2010.

\bibitem{scholkopf2001generalized}
B.~Sch{\"o}lkopf, R.~Herbrich, and A.~J. Smola, ``A generalized representer theorem,'' in \emph{International Conference on Computational Learning Theory}.\hskip 1em plus 0.5em minus 0.4em\relax Springer, 2001, pp. 416--426.

\bibitem{kadri2016operator}
H.~Kadri, E.~Duflos, P.~Preux, S.~Canu, A.~Rakotomamonjy, and J.~Audiffren, ``Operator-valued kernels for learning from functional response data,'' \emph{Journal of Machine Learning Research}, vol.~17, no.~20, pp. 1--54, 2016.

\bibitem{maddalena2021deterministic}
E.~T. Maddalena, P.~Scharnhorst, and C.~N. Jones, ``Deterministic error bounds for kernel-based learning techniques under bounded noise,'' \emph{Automatica}, vol. 134, p. 109896, 2021.

\end{thebibliography}
